\documentclass[12pt, oneside]{article}   	
\usepackage[a4paper]{geometry}                		
\usepackage{amsthm,amsmath,amssymb}
\usepackage{ mathrsfs }
\usepackage{array}
\usepackage{amssymb}
\usepackage{amsfonts,latexsym,amsthm,amssymb,amsmath,amscd,euscript}
\usepackage{framed}
\usepackage{fullpage}
    \usepackage{tikz-cd}

\usepackage{graphicx}

\usepackage[colorlinks=true,citecolor=black,linkcolor=black,urlcolor=blue]{hyperref}

\theoremstyle{plain}
\newtheorem{theorem}{Theorem}
\newtheorem{lemma}[theorem]{Lemma}
\newtheorem{corollary}[theorem]{Corollary}
\newtheorem{proposition}[theorem]{Proposition}

\newtheorem{claim}[theorem]{Claim}

\theoremstyle{definition}
\newtheorem{definition}[theorem]{Definition}

\theoremstyle{remark}

\newcommand{\BR}{\mathbb R}

\newcommand{\MA}{\mathcal{A}}
\newcommand{\MB}{\mathcal{B}}
\newcommand{\MC}{\mathcal{C}}
\newcommand{\MF}{\mathcal{F}}
\newcommand{\MW}{\mathcal{W}}

\newcommand{\BN}{\mathbb N}
\newcommand{\ra}{\rightarrow}

\newcommand{\E}{\mathbb{E}}
\newcommand{\ep}{\epsilon}
\title{\bf Coloring Chains for Compression with Uncertain Priors}

\author{Noah Golowich\\
\small Harvard University\\[-0.8ex] 
\small Cambridge, MA, U.S.A.\\
\small\tt ngolowich@college.harvard.edu\\
}

\date{{\today}}

\begin{document}

\maketitle


\begin{abstract}
Haramaty and Sudan considered the problem of transmitting a message between two people, Alice and Bob, when Alice's and Bob's priors on the message are allowed to differ by at most a given factor. To find a deterministic compression scheme for this problem, they showed that it is sufficient to obtain an upper bound on the chromatic number of a graph, denoted $U(N,s,k)$ for parameters $N,s,k$, whose vertices are nested sequences of subsets and whose edges are between vertices that have similar sequences of sets. In turn, there is a close relationship between the problem of determining the chromatic number of $U(N,s,k)$ and a local graph coloring problem considered by Erd\H{o}s et al. We generalize the results of Erd\H{o}s et al.~by finding bounds on the chromatic numbers of graphs $H$ and $G$ when there is a homomorphism $\phi :H\ra G$ that satisfies a nice property. We then use these results to improve upper and lower bounds on $\chi(U(N,s,k))$.

\end{abstract}

\section{Introduction}
\label{sec:introduction}
 We consider the following graph coloring problem. For a positive integer $N$, a {\it chain} of {\it length} $f$ and {\it size} $s$ is a nested sequence of sets $A_0 \subseteq A_1 \subseteq \cdots A_f \subseteq [N]$ with $|A_0| = 1$ and $|A_f| = s$. We denote such a chain by $\langle A_0, A_1, \ldots, A_f \rangle$; if $\alpha$ is the single element of $A_0$, we will also write $\langle \alpha, A_1, \ldots, A_f \rangle$. For a chain $\mathcal A$, given by $\langle A_0, \ldots, A_f \rangle$, $S^1(\mathcal A)$ is defined \cite{haramaty_deterministic_2016} to be the set of all chains $\langle B_0, \ldots, B_{f-1} \rangle$ such that for $0 \leq i \leq f-1$, $A_{i-1} \subseteq B_i \subseteq A_{i+1}$, where $A_{-1} = \emptyset$. Next, for a positive integer $j$ and a positive real number $x$, $\log^{(j)}x$ denotes the base-2 logarithm function iterated $j$ times. Moreover, $\log^*x$ denotes the minimum $j$ such that $\log^{(j)}x \leq 1$.
 Haramaty and Sudan \cite{haramaty_deterministic_2016} showed that for any $k,s \in \BN$ the set of all chains of length $2k$ and size at most $s$ can be colored with $2^{6(s+1)} \cdot \log^{(k)} N$ colors so that for chains $\mathcal A, \mathcal A'$ in this set with $S^1(\mathcal A) \cap S^1(\mathcal A') \neq \emptyset$ and $A_0 \neq A_0'$, $\mathcal A$ and $\mathcal A'$ receive different colors. As we will show in Section \ref{sec:chaingraphs}, this result is equivalent to the fact that the following graph, which we denote by $U(N,s,k)$, has a proper coloring with $2^{6(s+1)} \cdot \log^{(k)} N$ colors: $V(U(N,s,k))$ is the set of all chains of length $k$ and size at most $s$, and
\begin{eqnarray}
E(U(N,s,k)) &=&\{(\langle \alpha, A_1, \ldots, A_k \rangle, \langle \beta, B_1, \ldots, B_k \rangle) \nonumber\\
&&: \alpha \neq \beta, \alpha \in B_1, \beta \in A_1, \ \ \forall 1 \leq i \leq k - 1 : A_i \subseteq B_{i+1}, B_i \subseteq A_{i+1}\}.\nonumber
\end{eqnarray}
In this paper we prove an upper bound on the chromatic number of $U(N,s,k)$ that improves the upper bound found by Haramaty and Sudan when $s,k$ are small compared to $N$, and we also prove a lower bound on the chromatic number of $U(N,s,k)$ that greatly improves upon previous lower bounds. Before doing so, we introduce some notation and explain the motivation behind determining the chromatic number of the graphs $U(N,s,k)$.

{\bf Notation.} We use the following conventions. For a chain $\mathcal A = \ \langle A_0, A_1, \ldots, A_f \rangle$, we let $sz(\MA)$ denote the size of $\mathcal A$, i.e.~$sz(\MA) = |A_f|$. For $N \in \BN$, we let $Chain(N)$ denote the set of all chains $\langle A_0, A_1, \ldots, A_f \rangle$ with $f \in \BN$ and $A_f \subset [N]$. Given a graph $G$, we let $\chi(G)$ denote the chromatic number of $G$, $V(G)$ denote the set of vertices of $G$, $E(G)$ denote the set of edges of $G$, and for $v \in V(G)$, $N(v)$ denote the set of neighbors of $v$ (excluding $v$ itself). All of our graphs have no loops; that is, for $(u,v) \in E(G)$, $u \neq v$. For $\delta \geq 0$, we let $N_\delta(v)$ denote the set of all vertices of distance at most $\delta$ from $v$ (so that, for instance, $N(v) \cup \{v\} = N_1(v)$). For sets $S,T$ and a map $f : S \ra T$, for a subset $H \subset S$, we let $f(H) = \{ f(s) : s \in H\}$. For a graph $G$ and a subset $T \subset V(G)$, we let $G[T]$ be the subgraph of $G$ induced by $T$. We use $\log$ to denote the logarithm base 2 and $\ln$ to denote the natural logarithm.

 \subsection{Motivation}
 \label{sec:motivation}
 The purpose in \cite{haramaty_deterministic_2016} of determining the chromatic number of the graphs described above was to solve the following compression problem: for some finite universe $U$, suppose that Alice is operating under the belief that a message $m$ is chosen from $U$ according to the probability distribution $P$, and that Bob operates under the belief that $m$ is chosen according to the distribution $Q$. Both Alice and Bob know that their distributions $P,Q$ are ``close'' in the sense that they know of some $\Delta  \geq 0$ such that $\max_{m \in U} \left( \max \left( \log_2 \frac{P(m)}{Q(m)}, \log_2 \frac{Q(m)}{P(m)}\right) \right) \leq \Delta$. The smallest such $\Delta$ for which this inequality holds is denoted $\delta(P,Q)$. For $m$ drawn from $U$ according to $P$ (written as $m \sim_P U$), Alice wishes to communicate $m$ to Bob using a number of bits that is as small as possible in expectation. 

Juba et al.~\cite{juba_compression_2011} considered a version of this question when Alice and Bob are allowed to share common random bits, and showed that if so, they can communicate with $H(P) + 2\Delta + O(1)$ bits. Haramaty and Sudan considered this problem when Alice and Bob are not allowed to share common random bits. To state their setup, we let $\mathcal P(U)$ be the space of probability distributions over $U$; an {\it uncertain deterministic compression scheme} is a pair of functions $E : \mathcal P(U) \times U \ra \{0,1\}^* \cup \{\perp\}$ and $D : \mathcal P(U) \times ( \{0,1\}^* \cup \{ \perp\}) \ra U \cup \{\perp\}$ such that for $P,Q \in \mathcal P(U)$ with $\delta(P,Q) \leq \Delta$, $\forall m \in U$, either $E(P,m) = \perp$ (which happens with some small probability) or $D(Q,E(P,m)) = m$. In other words, with high probability $E(P,m) \neq \perp$, meaning that the encoding $E$ does not fail, and Bob can use the decoding function $D$ to recover the message $m$. Moreover, we want the expected length of the encoding $\E_{m \sim_P U} |E(P,m)|$ to be small. If the probability that $E(P,m) = \perp$ is 0, then the compression scheme is said to have no error.

The following compression scheme was introduced in \cite{haramaty_deterministic_2016}: if we let $N = |U|$, then given $P \in \mathcal P(U)$, $m \in U$, let $r = \lfloor - \log_2 P(m) \rfloor$ and $f = 2 \lfloor \log^* N\rfloor - 1$. 
 We now define a chain $\mathcal A$ of length $f$ by setting $A_0 = \{m\}$ and for $1 \leq k \leq f$, $A_k = \{ m' \in [N] \ : \  |\log_2(1/P(m')) - r| \leq k\Delta + 1 \}$. Since $P,Q$ are $\Delta$-close, Bob knows that $-\log_2 Q(m)$ is within $\Delta+1$ of $r$, and in general, for each $m' \in A_k$ with $k \geq 1$, that $-\log_2 Q(m')$ is within $(k+1)\Delta +1$ of $r$. Bob now constructs a chain $\mathcal B$, given by $\langle B_0, \ldots, B_{f-1} \rangle$, such that $B_0 = \{w\}$ for some $w$ with $|\log_2(1/Q(w)) - r| \leq \Delta + 1$, and for $1 \leq k \leq f-1$, $B_k = \{ m' \ : \ |\log_2(1/Q(m')) - r | \leq (k+1)\Delta + 1 \}$. Finally, Bob finds a chain $\mathcal A'$ of length $f$ and size at most $s$ such that $\mathcal B \in S^1(\mathcal A')$. 

 As $\mathcal B \in S^1(\mathcal A)$ as well, we have that $S^1(\mathcal A) \cap S^1(\mathcal A') \neq \emptyset$. Hence, if the set of chains of length $f$ and size at most $s$ can be colored so that for chains $\mathcal A, \mathcal A'$ in this set with $S^1(\mathcal A) \cap S^1(\mathcal A') \neq \emptyset$ and $A_0 \neq A_0'$, $\mathcal A$ and $\mathcal A'$ receive different colors, then Bob can recover the original message $m$ if Alice transmits the color of $\mathcal A$ (along with the integers $s$ and $r$). In particular, Bob only needs to find a chain $\mathcal A'$ as above that has the same color as $\mathcal A$, and then the single element of $A_0'$ is guaranteed to be $m$ \cite{haramaty_deterministic_2016}. Therefore, to minimize the expected length of the encoding, our goal is to color the set of such chains with as few colors as possible subject to the coloring condition above. Recall from above that Haramaty and Sudan showed that the set of chains of length $2k$ and size at most $s$ can be colored with at most $2^{6(s+1)} \cdot \log^{(k)} N$ colors in this way. This leads to an expected length of $2^{\frac{H(P)}{\epsilon} + 2\Delta \log^*N + O(1)}$ (for an error rate of at most $\epsilon$), which is not quite constant in $N$. In order to achieve an encoding of constant size, one possibility is to reduce the number of colors $2^{O(s)} \log^{(k)}N$ to $O(\log^{(k)}N)$ for all $k \leq \log^*N - c$, for some absolute constant $c$. Using this upper bound for $k = \log^*N - c$ immediately gives a constant number of colors, which corresponds to an encoding of constant size. Determining whether or not this is possible motivates our work. 

\subsection{Overview of results}

In Section \ref{sec:indepuppbnd}, we show that the upper bound of $2^{O(s)} \cdot \log^{(k)}N$ on the chromatic number of $U(N,s,k)$ can be improved to $2^{O(2^{2s})} \cdot \log^{(2k)}N$. For any $\ep > 0$, we can further improve the bound to $2^{O(2^{(1+\ep)s})} \cdot \log^{(2k)}N$ for a certain subgraph of $U(N,s,k)$ whose vertices are chains $\langle \alpha, A_1, \ldots, A_\delta \rangle$ where $|A_i|$, $1 \leq i \leq \delta$, grows exponentially with $i$. To obtain these upper bounds, we prove a result relating the chromatic numbers of graphs $G,H$ when there is a graph homomorphism $\phi : H \ra G$:
\begin{theorem}
\label{thm:homoupp}
If  $\chi(G) > 2$, $\phi : H \ra G$ is a graph homomorphism, and $r \in \BN$ such that $|\phi(N(v))| \leq r$ for each $v \in H$, then $\chi(H) \leq \lceil 2^r \log\log \chi(G) \rceil$.
\end{theorem}
Haramaty and Sudan \cite{haramaty_deterministic_2016} showed that in the context of Theorem \ref{thm:homoupp}, we have that $\chi(H) \leq 2r(r+1) \log \chi(G)$. Thus Theorem \ref{thm:homoupp} is an improvement when $\chi(G)$ is large compared to $r$. Given an arbitrary graph $G$, we may construct a graph $H$ and a homomorphism $\phi$ as in Theorem \ref{thm:homoupp} as follows: we let $V(H) = \{\langle v, S \rangle\}_{v \in V(G), S \subset N(v) \cup \{ v\}, v \in S, |S| \leq r+1}$, and a pair $(\langle v, S\rangle, \langle u, T \rangle) \in E(H)$ if and only if $v \in T$ and $u \in S$ and $u \neq v$. We will call $H$ the {\it restricted neighborhood graph} of $G$ and write $H = RN(G)$. Next, we define the homomorphism $\phi: RN(G) \ra G$ that maps $\langle v, S\rangle$ to $v$. It is clear that $\phi$ is indeed a homomorphism, and moreover that $|\phi(N(\langle v, S \rangle))| \leq r$ for all $\langle v, S \rangle \in V(H)$, as each element of $\phi(N(\langle v, S \rangle))$ must be in $S - \{v\}$.

In Section \ref{sec:lowbnd}, we prove a lower bound of $2^{\frac{s - 2k - 4}{2}} \cdot \log^{(2k)}N + o(1)$ on the chromatic number of $U(N,s,k)$ as $N \ra \infty$. To establish this bound, we prove a result that is similar in nature to Theorem \ref{thm:homoupp}, except it provides a lower bound on $\chi(H)$ in terms of $\chi(G)$. To state this result, we define a graph homomorphism $\phi :H \ra G$ to be {\it complete} if it satisfies the following property: for any $x, z \in V(H)$, if $\phi(z) \in \phi(N(x))$ and $\phi(x) \in \phi(N(z))$, then $(x,z) \in E(H)$. Then we have:
\begin{theorem}
\label{thm:homolow}
Suppose $\phi : H \ra G$ is a complete graph homomorphism such that for any $w \in G$ and neighbors $u_1, \ldots, u_r$ of $w$, there is some $v \in V(H)$ such that $u_1, \ldots, u_r \in \phi(N(v))$ and $\phi(v) = w$. Then if $\chi(H) = n$, we have that $\chi(G) \leq 2^{2n + 2^{n/2^{r-2}}}$.
\end{theorem}

Together, Theorems \ref{thm:homoupp} and \ref{thm:homolow} show that if $\phi : H \ra G$ is a complete graph homomorphism and $r_1,r_2 \in \BN$ such that $|\phi(N(v))| \leq r_1$ for each $v \in V(H)$ and for any $w \in G$ and neighbors $u_1, \ldots, u_{r_2}$ of $w$, there is some $v \in V(H)$ such that $u_1, \ldots, u_{r_2} \in \phi(N(v))$ and $\phi(v) = w$, then
\begin{equation}
\label{eq:r1r2}
2^{2^{\chi(H)/2^{r_1}}} \leq \chi(G) \leq 2^{2\chi(H) + 2^{\chi(H)/2^{r_2-2}}}.
\end{equation}
Note that we will always have $r_1 \geq r_2$ in such a scenario.

Note that for a graph $G$, the homomorphism $\phi : RN(G) \ra G$ as described above is complete: if $\langle v, S \rangle, \langle u, T \rangle \in V(RN(G))$, and $\phi(\langle v, S \rangle) \in \phi(N(\langle u,T\rangle))$ and $\phi(\langle u,T) \in \phi(N(\langle v,S))$, then $v \in T$ and $u \in S$, but $u \neq v$. This clearly implies that $(\langle v,S \rangle, \langle u,T \rangle) \in E(RN(G))$. Note also that for any $w \in G$ and neighbors $u_1, \ldots, u_r$, then the vertex $\langle w, \{ u_1, \ldots, u_r, w\}\rangle$ has the property that $u_1, \ldots, u_r \in \phi(\langle w, \{u_1, \ldots, u_r, w\} \rangle )$ and $\phi(\langle w, \{u_1, \ldots, u_r, w\}\rangle) = w$. Therefore, with $H = RN(G)$, it follows from Theorems \ref{thm:homoupp} and \ref{thm:homolow} that (\ref{eq:r1r2}) holds with $r_1=r_2=r$.

\section{Graph Independence and Upper Bounds}
\label{sec:indepuppbnd}
Many of the results presented in this section are generalizations of analogues proven by Erd\H{o}s et al.~in \cite{erdos_coloring_1986}, and which were discovered independently in \cite{szegedy_locality_1993}. As such, we will make a change in notation and write $U(m,R,\delta)$ instead of $U(N,s,k)$ to be consistent with the notation of \cite{erdos_coloring_1986}. Erd\H{o}s et al.~considered the following question: we say that a graph $G$ has a local $(m,R)$-coloring \cite{erdos_coloring_1986,korner_local_2005,simonyi_local_2006,simonyi_directed_2010,szegedy_locality_1993} if it has a proper coloring with $m$ colors that uses at most $T$ colors in the neighborhood of each vertex (including the vertex itself). Then given that $G$ has a local $(m,R)$-coloring, how can the chromatic number of $G$ be bounded above? Erd\H{o}s et al.~obtained nearly tight bounds on the answer to this question for values of $m,R$ in certain ranges. Many of the results presented below make progress towards answering the following generalization of this question, which Erd\H{o}s et al.~also stated (\cite{erdos_coloring_1986}, Definitions 4.1 and 4.2), and which was also raised by Szegedy and Vishwanathan \cite{szegedy_locality_1993}: we say that $G$ has a local $(m,R,\delta)$-coloring if it has a proper coloring with $m$ colors that uses at most $R$ colors in the distance-$\delta$ neighborhood $N_\delta(v)$ of each vertex $v$. Then for $\delta > 1$, by how much can the upper bound on the chromatic number of $G$ be improved (since if $G$ has a local $(m,R,\delta)$ coloring, then it certainly has a local $(m,R)$-coloring)?

For a graph $G$, we will consider in this section collections $\mathcal F_G$ of pairs $(v,S)$ that satisfy $v\in V(G)$, $S \subset V(G)$, and for each $u \in S$, $(v,u) \in E(G)$. We will be particularly interested in such collections $\mathcal F_G$ that are defined as follows:
for a graph $H$, and a graph homomorphism $\phi : H \ra G$, we define the collection of tuples $\mathcal F_{G,\phi}$ as follows:
$$
\mathcal F_{G,\phi} = \left\{ (\phi(v), \phi(N(v))) \right\}_{v \in H}.
$$

We first make the following definition pertaining to such collections $\mathcal F_G$, which generalizes Definition 4.5 (as well as Definition 1.4) in \cite{erdos_coloring_1986}, as well as an analogous definition in \cite{szegedy_locality_1993}. In doing so, we identify the vertices in $V(G)$ with the integers $\{1, 2, \ldots, |V(G)|\}$. Also, for a set $S \subset V(G)$, we let $\min S = \min_{u \in S} u$.

\begin{definition}
Consider a graph $G$, a collection $\mathcal F_G$ as above, and $n \in \BN$. Then the system of sets $\{A_{u,v}\}_{1 \leq u < v \leq |V(G)|, (u,v) \in E} \subseteq \mathcal P([n])$ is {\it $(G,n,\mathcal F_G)$-independent}, if, for any $(v,S) \in \mathcal F_G$ with $v > \min S$,
$$
\bigcap_{u < v, u \in S, (u,v) \in E} A_{u,v} - \bigcup_{w > v, w \in S, (v,w) \in E} A_{v,w} \neq \emptyset,
$$
and for any $(v,S) \in \mathcal F_G$ with $v < \min S$, we have
$$
[n] - \bigcup_{w > v, w\in S, (v,w) \in E} A_{v,w} \neq \emptyset.
$$
\end{definition}

Now we prove two lemmas which establish a link between the existence of $(G,n,\mathcal F_G)$-independent systems and the chromatic number of graphs, which generalize Lemma 4.4 (as well as Lemma 1.2) in \cite{erdos_coloring_1986}.
\begin{lemma}
\label{lem:colorupbound}
For a graph homomorphism $\phi : H \ra G$, suppose that there is a $(G, n, \mathcal F_{G,\phi})$-independent system. Then $\chi(H) \leq n$.
\end{lemma}
\begin{proof}
Let $\{A_{u,v}\}_{1 \leq u < v \leq |V(G)|,(u,v) \in E(G)}$ be a $(G,n,\mathcal F_{G,\phi})$-independent system. 
Take any vertex $x \in V(H)$, and if $\phi(x) > \min \phi(N(x))$ and $N(x)$ is nonempty, define
$$
g(x) = \min \left\{ \bigcap_{u \in \phi(N(x)), u < \phi(x)} A_{u,\phi(x)} - \bigcup_{v \in \phi(N(x)),v > \phi(x)} A_{\phi(x),v}\right\},
$$
where the set on the right hand side of the above equation is nonempty by $(G,n,\mathcal F_{G,\phi})$-independence. If $\phi(x) < \min(\phi(N(x))$ or $N(x)$ is empty, then define $$g(x) = \min\left\{ [n] - \bigcup_{v \in \phi(N(x)),v>\phi(x)} A_{\phi(x),v} \right\}.$$ We claim that $g$ is a proper coloring of $H$. To see this, take 2 vertices $x,y \in V(H)$ with $(x,y) \in E(H)$, and suppose without loss of generality that $\phi(x) < \phi(y)$. The fact that $\phi(x) \in \phi(N(y))$ implies that $g(y) \in A_{\phi(x),\phi(y)}$. 
We also have that $\phi(y) \in \phi(N(x))$, so $g(x) \not \in A_{\phi(x),\phi(y)}$, which implies that $g(x) \neq g(y)$, as desired.
\end{proof}

\begin{lemma}
\label{lem:colorlobound}
If $\chi(H) \leq n$, and $\phi : H \ra G$ is a complete graph homomorphism, then there is a
$(G,n,\mathcal F_{G,\phi})$-independent system.
\end{lemma}

\begin{proof}
Suppose $g$ is a proper $n$-coloring of $H$. We define a $(G,n,\mathcal F_{G,\phi})$-independent system $\{A_{u,v}\}$ as follows. For $(u,v) \in E(G)$, with $1 \leq u < v \leq |V(G)|$, we let
$$
A_{u,v} = \left\{ g(x) \ \  :\ \  x \in V(H), \ \phi(x) = v,\ u \in \phi(N(x))\right\}.
$$
We claim that this system is $(G,n,\mathcal F_{G,\phi})$-independent. For suppose not; there are two possibilities:
\begin{enumerate}
\item There is some $x \in V(H)$, $\phi(x) > \min \phi(N(x))$, such that
$$
\bigcap_{u \in \phi(N(x)), u < \phi(x)} A_{u,\phi(x)} - \bigcup_{v \in \phi(N(x)), v > \phi(x)} A_{\phi(x),v} = \emptyset.
$$
Let $\xi = g(x)$. We claim that
$$
\xi \in \bigcap_{u \in \phi(N(x)), u < \phi(x)} A_{u,\phi(x)}.
$$
To see that this is the case, note that
$$
A_{u,\phi(x)} = \{ g(y) \ \ : \ \ y \in V(H), \phi(y) = \phi(x), u \in \phi(N(y))\},
$$
so that for each $u \in \phi(N(x))$ with $u < \phi(x)$, we may simply choose $y = x$, and always have that $u \in \phi(N(y))$. Since $\phi(x) > \min \phi(N(x))$, there always exists at least one such $u$.
\item There is some $x\in V(H)$, $\phi(x) < \min \phi(N(x))$, such that
$$
[n] - \bigcup_{v \in \phi(N(x)), v>\phi(x)} A_{\phi(x),v} = \emptyset.
$$
Again, let $\xi = g(x)$, so that $\xi \in [n]$.
\end{enumerate}

 In both cases above, there must exist $v > \phi(x)$, with $v \in \phi(N(x))$, such that $\xi \in A_{\phi(x),v}$. In particular, this means that $\xi = g(z)$, for some $z \in V(H)$ with $\phi(z) = v$ and $\phi(x) \in \phi(N(z))$. 
Since $\phi$ is complete, this immediately implies that $(z,x) \in E(H)$, which is a contradiction to the fact that both $x$ and $z$ are colored $\xi$. This completes the proof.
\end{proof}
Lemmas \ref{lem:colorupbound} and \ref{lem:colorlobound} immediately imply the following:
\begin{proposition}
\label{prop:indepequiv}
If $\phi : H \ra G$ is a complete graph homomorphism, then $\chi(H) \leq n$ if and only if there exists a $(G,n,\mathcal F_{G,\phi})$-independent system.
\end{proposition}

From Proposition \ref{prop:indepequiv}, in order to prove upper bounds on the chromatic number of a graph $H$, we need to prove the existence of $(G,n,\mathcal F_{G,\phi})$-independent systems for appropriate choices of $\phi,G$. To do so, we will use a result of Kleitman and Spencer \cite{kleitman_families_1973} on the existence of families of independent sets.

\begin{definition}[Kleitman and Spencer, \cite{kleitman_families_1973}]
If $S$ is an $n$-element set, then the $k$ subsets $A_1, \ldots, A_k \subseteq S$ are defined to be {\it $k$-independent} if all $2^k$ intersections $\cap_{j=1}^n B_j$ (where $B_j$ can be either $A_j$ or $\bar A_j$, and where $A_j$ and $\bar A_j$ do not both appear among the $B_j$), are nonempty. More generally, the $m$ subsets $A_1, \ldots, A_m \subseteq S$, for $m \geq k$, are $k$-independent if each $k$-element subset of $\{A_1, \ldots, A_m \}$ is $k$-independent.
\end{definition}
Another way of stating the independence of $A_1, \ldots, A_k$ is that all $2^k$ portions of the Venn diagram relating $A_1, \ldots, A_k$, are nonempty. Kleitman and Spencer defined $f(n,k)$ to the the maximum size of a collection of a $k$-independent collection of subsets of an $n$-element set. Their main result was:
\begin{theorem}[Kleitman and Spencer, \cite{kleitman_families_1973}]
\label{thm:kindep}
We have:
$$
f(n,2) = {n-1 \choose \lfloor n/2 \rfloor -1},
$$
and there are absolute constants $d_1 \geq 1, d_2$ so that for each fixed $k \geq 3$, there is a sufficiently large $N_k$, so that for all $n \geq N_k$,
$$
2^{d_1n2^{-k}/k} \leq f(n,k) \leq 2^{d_2n2^{-k}}.
$$
\end{theorem}

The proof of the lower bound for $f(n,k)$ in \cite{kleitman_families_1973} was probabilistic, but an explicit construction was later found in \cite{alon_explicit_1986}. We now use the existence of independent collections of sets as guaranteed to exist in Theorem \ref{thm:kindep} to prove the existence of $(G,n,\mathcal F_G)$-independent systems for appropriate choices of $G,\mathcal F_G$ in Lemma \ref{lem:makegraphindep} below. The proof of this lemma is similar to that of Theorem 2.4 in \cite{erdos_coloring_1986}.
\begin{lemma}
\label{lem:makegraphindep}
If there is an $r$-independent collection of $k$ subsets of an $n$-element set, and $G$ is a graph with $\chi(G) = h \leq 2^k$, and $\mathcal F_G$ is a collection of pairs $(v,S)$ (with $v \in V(G)$, $S \subset V(G)$ and $S$ only contains neighbors of $v$) where each such pair has $|S| \leq r$, then there is a $(G,n,\mathcal F_G)$ independent system.
\end{lemma}
\begin{proof}
Suppose $\chi : G \ra [h]$ is a proper coloring of $G$, and by re-ordering the vertices of $G$ we can assume without loss of generality that $\chi$ respects the ordering of $V$; that is, for $u < v \in V(G)$, we have that $\chi(u) \leq \chi(v)$.

Next suppose that we have an $n$-element set $Q$, and subsets $A_1, \ldots, A_k \subseteq Q$ that are $r$-independent, for some $r \leq k$. Moreover recall that $h \leq 2^k$. Let $\mathscr C = \{A_1, \ldots, A_k\}$, and suppose that $\{Y_i : 1 \leq i \leq 2^k\}$ is an enumeration of the power set $\mathcal P(\mathscr C)$ with $|Y_i| \leq |Y_j|$ for $i < j$. Define the system of subsets
$$
\mathcal T = \left\{ A_{u,v} : 1 \leq u < v \leq |V(G)|, (u,v) \in E(G) \right\},
$$
by letting $A_{u,v} \in Y_{\chi(v)} - Y_{\chi(u)}$, where we have used the fact that $\chi(v) > \chi(u)$ for $v > u$ such that $(u,v) \in E$ (so in particular, we cannot have that $\chi(u) = \chi(v)$). We are also using the fact here that each $\chi(u) \leq 2^k$, which follows from $h \leq 2^k$. We claim that the collection $\mathcal T$ is $(G,n,\mathcal F_G)$-independent as long as for each $(v,S) \in \mathcal F_G$, we have $|S| \leq r$. To see this, note that for any $(v,S) \in \mathcal F_G$, we have that
\begin{eqnarray}
&& \bigcap_{u < v, u \in S, (u,v) \in E} A_{u,v} - \bigcup_{w > v, w \in S, (w,v) \in E} A_{v,w}\nonumber\\
\label{eq:rindep}
&=& A_{q_1} \cap \cdots \cap A_{q_s} \cap \bar A_{p_1} \cap \cdots \cap \bar A_{p_t},
\end{eqnarray}
for $1 \leq q_1, \ldots, q_s, p_1, \ldots, p_t \leq k$, and $s+t \leq |S| \leq r$. (If $v < \min S$, then the relevant quantity is $\bar A_{p_1} \cap \cdots \cap \bar A_{p_t}$, and $s = 0$.) Note that each of $A_{q_1}, \ldots, A_{q_s}$ are equal to one of $A_{u,v}$, and that each of $A_{p_1}, \ldots,  A_{p_t}$ are equal to one of $A_{v,w}$. For any $A_{u,v}$, we have that $A_{u,v} \in Y_{\chi(v)}$, and for any $A_{v,w}$ we have that $A_{v,w} \not \in Y_{\chi(v)}$. Therefore, we have that $A_{u,v} \neq A_{v,w}$ for all valid choice of $u,w$. Therefore, by $r$-independence of the collection $\mathscr C = \{A_1, \ldots, A_k\}$, we have that (\ref{eq:rindep}) is nonempty. This implies that $\mathcal T$ is $(G,n,\mathcal F_G)$-independent.
\end{proof}

Theorem \ref{thm:homoupp} now follows as an immediate consequence of Lemma \ref{lem:makegraphindep} and Theorem \ref{thm:kindep}:
\begin{proof}[Proof of Theorem \ref{thm:homoupp}]
Given $\phi : H \ra G$ with $|\phi(N(v))| \leq r$ for each $v \in H$, note that the collection $\mathcal F_{G,\phi}$ satisfies $|S| \leq r$ for each $(v,S) \in \mathcal F_{G,\phi}$. Next, by Theorem \ref{thm:kindep}, there is an $r$-independent collection of $2^{n2^{-r}/r}$ subsets of an $n$-element set. Therefore, by Lemma \ref{lem:makegraphindep}, as long as $\chi(G) \leq 2^{2^{n2^{-r}/r}}$, we have that a $(G,n,\mathcal F_{G,\phi})$-independent system exists. Lemma \ref{lem:colorupbound} then implies that $\chi(H) \leq n$. Note that $\chi(G) \leq 2^{2^{n2^{-r}/r}}$ is equivalent to $n \geq r2^r \log \log \chi(G)$, which implies that $\chi(H) \leq \lceil r2^r \log \log \chi(G) \rceil$.
\end{proof}

\subsection{Chain graphs: basic facts}

Before deriving our upper bounds on $\chi(U(m,R,\delta))$ for various choices of $m, R, \delta$, we first establish some basic facts about the chain graphs $U(m,R,\delta)$. The first result, Proposition \ref{prop:2colorings} below, explains how the result of Haramaty and Sudan \cite{haramaty_deterministic_2016} implies an upper bound on $\chi(U(m,R,\delta))$ (which is weaker than ours). Haramaty and Sudan \cite{haramaty_deterministic_2016} showed that for $m,R,\delta \in \BN$, the set $V(U(m,R,2\delta))$ can be colored with at most $2^{6(R+1)}\log^{(\delta)}m$ colors, such that for any two chains $\mathcal A = \langle \alpha, \ldots, A_{2\delta} \rangle, \mathcal B = \langle \beta, \ldots, B_{2\delta}\rangle$ in this set, if $S^1(\mathcal A) \cap S^1(\mathcal B) \neq \emptyset$ and $\alpha \neq \beta$, then $\mathcal A$ and $\mathcal B$ are colored by different colors. By Proposition \ref{prop:2colorings}, this implies that there is a proper vertex coloring of the graph $U(m,R,\delta)$ with at most $2^{(6(R+1))} \cdot \log^{(\delta)}m$ colors. 

\begin{proposition}
  \label{prop:2colorings}
Let $R, \delta, c \in \BN$. Then the following two statements are equivalent:
\begin{enumerate}
\item There exists a $c$-coloring of the set of all chains in $Chain(m)$ that have size at most $s$ and length $2\delta$ such that for any two chains $\mathcal A = \langle \alpha, \ldots, A_{2\delta} \rangle, \mathcal B = \langle \beta, \ldots, B_{2\delta}\rangle$ in this set, if $S^1(\mathcal A) \cap S^1(\mathcal B) \neq \emptyset$ and $\alpha \neq \beta$, then $\mathcal A$ and $\mathcal B$ are colored by different colors.
\item There is a proper $c$-coloring of $U(m,R,\delta)$.
\end{enumerate}
\end{proposition}
\begin{proof}
We first suppose that (1) is true, and construct a proper $c$-coloring of $U(m,R,\delta)$. In particular, for any vertex, say $\mathcal A = \langle \alpha, A_1, \ldots, A_\delta \rangle$, we may give it the color of $\langle \alpha, A_1, A_1, A_2, A_2, \ldots, A_\delta, A_\delta \rangle$, which is a chain of length $2\delta$ and size at most $R$. To see that this is a proper coloring, consider an edge $(\langle \alpha, A_1, \ldots, A_\delta \rangle, \langle \beta, B_1, \ldots, B_\delta \rangle)$ of $U(m,R,\delta)$. Then
\begin{eqnarray}
&&\langle \alpha, \{\alpha\} \cup \{\beta\}, A_1 \cup B_1, A_1 \cup B_1, \ldots, A_{\delta-1} \cup B_{\delta-1}, A_{\delta-1} \cup B_{\delta-1} \rangle\nonumber\\
&\in& S^1(\langle \alpha, A_1, A_1, A_2, A_2, \ldots, A_\delta, A_\delta \rangle) \cap S^1(\langle \beta, B_1, B_1,B_2,B_2, \ldots, B_\delta, B_\delta \rangle).\nonumber
\end{eqnarray}
Since also $\alpha \neq \beta$, by the definition of the coloring in statement (1), it follows that $\langle \alpha, A_1, \ldots, A_\delta \rangle$ and $\langle \beta, B_1, \ldots, B_\delta \rangle$ receive different colors.

Next suppose that we are given a proper $c$-coloring of $U(m,R,\delta)$. For each chain $\mathcal A = \langle \alpha, A_1, \ldots, A_{2\delta} \rangle$ of size at most $R$ and length $2\delta$, we give $\mathcal A$ the color of $\langle \alpha, A_2, A_4, \ldots, A_{2\delta} \rangle\\ \in V(U(m,R,\delta))$. To see that this coloring satisfies the condition in (1), suppose that for chains $\mathcal A = \langle \alpha, A_1, \ldots, A_{2\delta} \rangle$ and $\mathcal B = \langle \beta, B_1, \ldots, B_{2\delta} \rangle$ of length $2\delta$ and size at most $R$, the chain $\mathcal C = \langle \gamma, C_1, \ldots, C_{2\delta-1} \rangle \in S^1(\mathcal A) \cap S^1(\mathcal B)$ and $\alpha \neq \beta$. Then $\alpha \in C_1 \subseteq B_2$, $\beta \in C_1 \subseteq A_2$, and for $1 \leq i \leq \delta-1$, $A_{2i} \subseteq C_{2i+1} \subseteq B_{2(i+1)}$ and $B_{2i} \subseteq C_{2i+1} \subseteq A_{2(i+1)}$. This implies that $\langle \alpha, A_2, \ldots, A_{2\delta} \rangle$ and $\langle \beta, B_2, \ldots, B_{2\delta} \rangle$ are adjacent in the graph $U(m,R,\delta)$, which implies that $\mathcal A$ and $\mathcal B$ indeed receive different colors this way.
\end{proof}


Our goal is to determine if it is possible to obtain some kind of bound on $\chi(U(m,R,\delta))$ that improves the bound $\chi(U(m,R,\delta)) \leq 2^{O(R)} \cdot \log^{(\delta)}m$ from \cite{haramaty_deterministic_2016}. To do so, we will use Theorems \ref{thm:homoupp} and \ref{thm:homolow} to reason about $\chi(U(m,R,\delta))$; first, though, we must establish an appropriate graph homomorphism $\phi$ used in those theorems. 

We define the map $\phi : V(U(m,R,\delta)) \ra V(U(m,R,\delta-1))$ by
\begin{equation}
  \label{eq:definephi}
\phi(\langle \alpha, A_1, \ldots, A_{\delta} \rangle) = \langle \alpha, A_1, \ldots, A_{\delta-1} \rangle.
\end{equation}
It is immediate that $\phi$ is a graph homomorphism. It is also complete:
\begin{lemma}
\label{lem:tset}
For any choice of $m,R,\delta$, the graph homomorphism $\phi$ defined in (\ref{eq:definephi}) is complete. 
\end{lemma}
\begin{proof}
  %
%
Consider any chains $\mathcal A, \mathcal B \in V(W(m,\sigma,\delta))$, and suppose that $\phi(\mathcal A) \in \phi(N(\mathcal B))$ and $\phi(\mathcal B) \in \phi(N(\mathcal A))$. Let us write $\mathcal A = \langle \alpha, A_1, \ldots, A_\delta \rangle$ and $\mathcal B = \langle \beta, B_1, \ldots, B_\delta \rangle$. Since $\phi(\mathcal A) \in \phi(N(\MB))$, we have that for some $\MC = \langle \gamma, C_1, \ldots, C_\delta \rangle \in N(\MB)$, $\phi(\MA) = \phi(\MC)$. In particular, this means that $\alpha = \gamma$ and $A_i = C_i$ for $1 \leq i \leq \delta-1$. Since $\MC \in N(\MB)$, it follows that $\alpha \in B_1, \beta \in A_1$, and for $1 \leq i \leq \delta - 2$, $A_i \subseteq B_{i+1}$ and $B_i \subseteq A_{i+1}$. This also gives us that $A_{\delta-1} \subseteq B_\delta$. In a symmetric manner, since $\phi(\mathcal B) \in \phi(N(\MA))$, we have that $B_{\delta-1} \subseteq A_{\delta}$. This implies that $(\mathcal A, \mathcal B) \in E(W(m,\sigma,\delta))$.
\end{proof}

The next lemma states that completeness respects restrictions to induced subgraphs. It will be useful when we prove lower bounds on the chromatic number of induced subgraphs of $U(m,R,\delta)$ in Section \ref{sec:lowbnd}. Given a graph homomorphism $\phi : H \ra G$ and an induced subgraph $H'$ of $H$, we will denote the restriction of $\phi$ to $H'$ by $\phi_{H'} : H' \ra G$. Moreover, in the proof of the below lemma, for a vertex $x \in V(H)$, we denote by $N_{H}(x)$ the neighborhood of $x$ in $H$ and by $N_{H'}(x)$ the neighborhood of $x$ in $H'$.
\begin{lemma}
  \label{lem:restindsub}
Suppose that $G,H$ are graphs and $H'$ is an induced subgraph of $H$. Suppose that $\phi : H \ra G$ is a complete graph homomorphism. Then the restriction $\phi_{H'} : H' \ra G$ is also complete.
\end{lemma}
\begin{proof}
Consider vertices $x,z \in V(H')$ such that $\phi_{H'}(x) \in \phi_{H'}(N_{H'}(z))$ and $\phi_{H'}(z) \in \phi_{H'}(N_{H'}(x))$. Our aim is to show that $(x,z) \in E(H')$. Since $N_{H'}(x) \subseteq N_H(x)$ and $N_{H'}(z) \subseteq N_H(z)$, it follows that $\phi(x) \in \phi(N_H(z))$ and $\phi(z) \in \phi(N_H(x))$. Since $\phi$ is complete, we have $(x,z) \in E(H)$ as a consequence. Since $H'$ is an induced subgraph and $(x,z) \in E(H')$, it follows that $(x,z) \in E(H')$, as desired.
\end{proof}

\section{Upper bound on chromatic number of chain graphs}
\label{sec:chaingraphs}
Now we use the results in the previous section to derive an upper bound on the chromatic number of $U(m,R,\delta)$, as well as an improved upper bound on a subgraph of $U(m,R,\delta)$ whose vertices are chains that grow exponentially in size (here recall that $m,R,\delta \in \BN$, where $m$ denotes the size of the universe, $R$ denotes that maximum size of the chains, and $\delta$ denotes the length of the chains). 
We begin with a small lemma that allows us to bound the ``$r$'' parameter in Theorem \ref{thm:homoupp}.
\begin{lemma}
  \label{lem:boundtsetgen}
For $\mathcal A \in V(U(m,R,\delta))$, we have that $| \phi(N(\MA)) | \leq 2^{2 \cdot sz(\MA)} \leq 2^{2 R}$.
\end{lemma}
\begin{proof}
  Write $\MA = \langle \alpha, A_1, \ldots, A_\delta \rangle$, and consider any $\langle \beta, B_1, \ldots, B_\delta \rangle \in \phi(N(\MA))$. We must have $\beta \in A_1$, $\alpha \in B_1$, $B_1 \subseteq A_2$, and $A_{i-1} \subseteq B_i \subseteq A_{i+1}$ for $2 \leq i \leq \delta - 1$, so the number of choices for $\langle \beta, B_1, \ldots, B_{\delta-1} \rangle$ is at most
  \begin{eqnarray}
     && |A_1| \cdot 2^{|A_2|-1} \cdot 2^{|A_3| - |A_1|} \cdot 2^{|A_4| - |A_2|} \cdots 2^{|A_\delta| - |A_{\delta-2}|}\nonumber\\
    & \leq & 2^{|A_1| + \cdots + |A_\delta| - (|A_1| + |A_2| + \cdots + |A_{\delta-2}|)}\nonumber\\
    & \leq & 2^{2 \cdot |A_\delta|} = 2^{2 \cdot sz(\MA)}\nonumber.
  \end{eqnarray}
  \end{proof}

  When $R \ll m$, Theorem \ref{thm:uppbndgen} improves the bound $\chi(U(m,R,\delta)) \leq 2^{O(R)} \cdot \log^{(\delta)}m$ of \cite{haramaty_deterministic_2016} to $\chi(U(m,R,\delta)) \leq 2^{O(2^{2R})} \cdot \log^{(2\delta)}m$.
  \begin{theorem}
    \label{thm:uppbndgen}
    If $\log^{(2\delta-2)} m \geq 2^{2^{2+2R}}$, then $\chi(U(m,R,\delta)) \leq 2^{2^{2+2R}} \cdot \log^{(2\delta)} m$.
  \end{theorem}
  \begin{proof}
    We use induction on $\delta$. For the base case $\delta = 0$, we have that $U(m,R,\delta) = K_m$, the complete graph on $m$ vertices. Then $\chi(K_m) = m = \log^{(0)}m$.

    Now suppose the result is true for $\delta - 1$. Let $G = U(m,R,\delta-1)$, so that $\chi(G) \leq 2^{2^{2 + 2R}} \cdot \log^{(2\delta - 2)} m$. Let $H = U(m,R,\delta)$. Consider the graph homomorphism $\phi : H \ra G$ defined in (\ref{eq:definephi}), which is complete by Lemma \ref{lem:tset}. Now consider any $(v, S) \in \MF_{G, \phi}$ (recall that $\MF_{G,\phi}$ is the set of all pairs $(\phi(\MA), \phi(N(\MA)))$, where $\MA \in V(U(m,R,\delta))$). By Lemma \ref{lem:boundtsetgen}, we have that $|S| \leq 2^{2R}$. By Theorem \ref{thm:homoupp}, we have that
    $$
    \chi(U(m,R,\delta)) \leq \lceil 2^{2R} 2^{2^{2R}} \log \log \chi(U(m,R,\delta-1)) \rceil.
    $$
    Since $2^{2^{2 + 2R}} \leq \log^{(2\delta - 2)}m$, the above equation implies that
    $$
\chi(U(m,R,\delta)) \leq \left\lceil 2^{2R} 2^{2^{2R}} (1 + \log^{(2\delta)} m) \right\rceil \leq 2^{2+2R} 2^{2^{2R}} \log^{(2\delta)}m \leq 2^{2^{2+2R}} \log^{(2\delta)}m.
$$
\end{proof}
We will occasionally write the upper bound in Theorem \ref{thm:uppbndgen} as $\chi(U(m,R,\delta)) \leq \\ 2^{2^{2+2R}} \log^{(2\delta)}m + o(1)$ (as $m \ra \infty$) since for fixed $R, \delta$, the upper bound holds for sufficiently large $m$.

Next we explain how to improve the upper exponent in the upper bound of Theorem \ref{thm:uppbndgen} from $2 + 2R$ to to $2 + R \cdot \frac{r}{r-1}$ for a particular subgraph of $U(m,R,\delta)$ defined by a parameter $r$. In particular, given $r, m, \sigma, \delta \in \BN$ with $r \geq 2$, we define the graph $W_r(m,\sigma,\delta)$ as follows:
$$
V = \{ \langle \alpha, A_1, \ldots, A_\delta \rangle : \alpha \in A_1 \subset A_2 \subset \cdots \subset A_\delta \subset [m], \ \ \forall 1 \leq i \leq \delta, |A_i| = r^{i-1} \cdot r^{\sigma}\},
$$
and 
\begin{eqnarray}
\label{eq:chainedge}
E &=&\{(\langle \alpha, A_1, \ldots, A_\delta \rangle, \langle \beta, B_1, \ldots, B_\delta \rangle) : \nonumber\\
&& \alpha \neq \beta, \alpha \in B_1, \beta \in A_1, \ \ \forall 1 \leq i \leq \delta - 1 : A_i \subseteq B_{i+1}, B_i \subseteq A_{i+1}\}.
\end{eqnarray}
We refer the reader to Table \ref{tab:cgs} for a summary of the chain graphs considered in this paper, including $W_r(m,\sigma,\delta)$. Note that for any vertex (chain) $\mathcal A \in V(W_r(m,\sigma,\delta))$, we have that $sz(\mathcal A) = r^{\sigma + \delta - 1}$, 
so $W_r(m,\sigma,\delta)$ is an induced subgraph of $U(m,r^{\sigma+\delta-1},\delta)$. 
Therefore, by Theorem \ref{thm:uppbndgen}, we have that $\chi(W_r(m,\sigma,\delta)) \leq 2^{2^{2 + 2 r^{\sigma + \delta - 1}}} \cdot \log^{(2\delta)}m$ for sufficiently large $m$. In Theorem \ref{thm:uppbndexp}, we prove that we can improve the upper bound to $2^{2^{2 + \frac{r+1}{r} \cdot r^{\sigma + \delta - 1}}}$.

\begin{table}[]
  \centering

\begin{tabular}{|>{\centering\arraybackslash}m{1.8cm}|>{\centering\arraybackslash}m{4.2cm}|>{\centering\arraybackslash}m{3.7cm}|>{\centering\arraybackslash}m{4.4cm}|}
    \hline
    Graph & Vertex growth condition & $\chi$ upper bound & $\chi$ lower bound \\\hline
   $U(m,R,\delta)$ &  $|A_\delta| \leq R$ & $2^{2^{2+2R}} \log^{(2\delta)}m + o(1)$ & $2^{\frac{R-2\delta-4}{2}} \log^{(2\delta)}m + o(1)$   \\\hline
   $W_r(m,\sigma,\delta)$ &  $|A_i| = r^{\sigma + i-1}, 1 \leq i \leq \delta$ & $2^{2^{2+r^{\sigma+\delta-2}(r+1)}}  \log^{(2\delta)}m$ &  $2^{\frac{r^{\sigma+\delta-2}(r-1)-5}{2}} \log^{(2\delta)} m  + o(1)$  \\\hline
   $Y(m,\delta)$ & $|A_i| = 2i+1$, $1 \leq i \leq \delta$   & $2^{2^{2 + 2(2\delta+1)}}\log^{(2\delta)}m$  & $\log^{(2\delta)}m$  \\\hline
   $Z(m,R,\delta)$ & $|A_i| = 2i+1$, $1 \leq i \leq \delta-1$, $|A_\delta| = R$& $2^{2^{2+2R}}\log^{(2\delta)}m + o(1)$& $2^{\frac{R-2\delta-4}{2}} \log^{(2\delta)}m + o(1)$\\\hline
\end{tabular}
\caption{Summary of chain graphs considered in this paper. In the second column (``Vertex growth condition'') a typical vertex of any of these graphs is denoted by $\langle \alpha, A_1, \ldots, A_\delta \rangle$, with $\alpha \in A_1 \subseteq \cdots \subseteq A_\delta \subseteq [m]$. In some cells, the best known known upper or lower bounds are determined by subgraph relations (for instance, $Y(m,\delta)$ is a subgraph of $U(m,2\delta+1,\delta)$, so $\chi(Y(m,\delta)) \leq \chi(U(m,2\delta+1,\delta))$).}
\label{tab:cgs}
\end{table}

\begin{theorem}
\label{thm:uppbndexp}
For $r, \sigma, \delta, m\in \BN$ with $r \geq 2$ and $\log^{(2\delta)}m\geq 1$, we have $\chi(W_r(m,\sigma,\delta)) \leq 2^{2^{2+r^{\sigma + \delta - 2}(r+1)}} \cdot \log^{(2\delta)}m$.
\end{theorem}
The proof of Theorem \ref{thm:uppbndexp} is very similar to that of Theorem \ref{thm:uppbndgen}, except it uses the exponential growth of the chains in $V(W_r(m,\sigma,\delta))$, and can be found in Appendix \ref{apx:pfs}.

\section{Lower bounds}
\label{sec:lowbnd}
In this section we prove lower bounds on the chromatic number of $U(m,R,\delta)$ by establishing lower bounds on the chromatic number of certain induced subgraphs of $U(m,R,\delta)$. The key ingredient to doing so is Theorem \ref{thm:homolow}, which we prove first. We state it below in a slightly different form:
\begin{theorem}
\label{thm:reclowbnd}
Suppose $\phi : H \ra G$ is a complete graph homomorphism such that for any $w \in G$ and neighbors $u_1, \ldots, u_r$ of $w$, there is some $(w,S) \in \mathcal F_{G,\phi}$ such that $u_1, \ldots, u_r \in S$. Then if $\chi(H) \leq n$, we have that $\chi(G) \leq 2^{2n + 2^{n/2^{r-2}}}$.
\end{theorem}
The proof of Theorem \ref{thm:reclowbnd} (equivalently, Theorem \ref{thm:homolow}) is similar to the proof of Theorem 2.3 in \cite{erdos_coloring_1986}.
The bulk of this proof is contained in Lemma \ref{lem:inductr}. 
\begin{lemma}
\label{lem:inductr}
Suppose $G$ is a graph and $\mathcal F_G$ is a collection of pairs $(v,S)$ (where $v \in V(G)$ and $S \subset V(G)$) such that for any $v \in G$ and neighbors $u_1, \ldots, u_r$ of $v$, there is some $(v,S) \in \mathcal F_{G,\phi}$ such that $u_1, \ldots, u_r \in S$. Also suppose that a $(G,n,\mathcal F_{G})$-independent system exists. Then there is a partition of $V(G)$ into $n'+1 \leq 2^{n-1}+1$ sets, say $T_1, T_2, \ldots, T_{n'}, T_{n'+1}$, and collections $\mathcal F_{G[T_1]}, \ldots, \mathcal F_{G[T_{n'}]}$ such that:
\begin{enumerate}
\item $T_{n'+1}$ is an independent set in $G$.
\item For $1 \leq i \leq n'$, 
there is a $(G[T_i],\lfloor n/2 \rfloor, \mathcal F_{G[T_i]})$-independent system.
\item For each $1 \leq i \leq n'$, for any $v \in T_i$, and neighbors $u_1, \ldots, u_{r-1}$ in $G[T_i]$, there is some $(v,S) \in \mathcal F_{G[T_i]}$ such that $u_1, \ldots, u_{r-1} \in S$.
\end{enumerate}
\end{lemma}
\begin{proof}
Associating $V(G)$ with $\{1, \ldots, |V(G)|\}$, let us denote a $(G,n,\mathcal F_{G})$-independent system by $\{A_{u,v} : 1 \leq u < v \leq |V(G)|, \ u,v \in V(G)\}.$ For $v \in V(G)$, and $A \subset [n]$ with $|A| \leq \lfloor n/2 \rfloor$, define $v$ to be of {\it type} $A$ if one of the statements below holds:
\begin{enumerate}
\item There exists $u < v$ with $A_{u,v} = A$, or
\item There exists $w > v$ with $A_{v,w} = [n] - A$.
\end{enumerate}
Note that for any vertices $u < v$ with $(u,v) \in E(G)$, either $v$ is of type $A_{u,v}$ or $u$ is of type $[n] - A_{u,v}$, where the former holds if $|A_{u,v}| \leq \lfloor n/2 \rfloor$, and the latter holds if $|A_{u,v}| \geq \lceil n/2 \rceil$. Therefore, if $v$ is not of type $A$ for any $A$, then each of its neighbors is of type $A$ for some $A$, meaning that the set of vertices that are not of type $A$ for any $A$ form an independent set in $G$; let this set be $T_{n'+1}$. The number of sets $A \subset [n]$ with $|A| \leq \lfloor n/2 \rfloor$ is at most $2^{n-1}$, meaning that if we index such sets by $A_1, A_2, \ldots, A_{n'}$ with $n' \leq 2^{n-1}$ we may let $T_i$ be the set of all vertices of type $A_i$ (if a vertex is of type $A$ for more than 1 set $A$, we pick $A$ arbitrarily).

We next define the collections $\mathcal F_{G[T_i]}$ as follows. Consider any pair $(v,S) \in \mathcal F_G$ that has the property that $v \in T_i$ for some $i$. There now is either some $u_v < v$ such that $A_{u_v,v} = A_i$ or some $u_v > v$ such that $A_{v,u_v} = [n] - A_i$. We will now include the pair $(v,S \cap T_i)$ in $\mathcal F_{G[T_i]}$ if and only if $u_v \in S$. In other words, we have
$$
\mathcal F_{G[T_i]} = \left\{ (v,S \cap T_i)  \ \ : \ \ (v,S) \in \mathcal F_G, \ v \in T_i, \ u_v \in S\right\}.
$$

Next, for $1 \leq i \leq n'$, consider the collection $\mathcal F_{G[T_i]}$, and pick any $v \in T_i$. 
Now consider any neighbors $u_1, \ldots, u_{r-1}$ of $v$ in $G[T_i]$. Since $v$ and $u_v$ are neighbors, we know that there exists some $(v,S) \in \mathcal F_G$ such that $u_1, \ldots, u_{r-1}, u_v \in S$. Therefore, letting $\hat S = S \cap T_i$, we have that $(v,\hat S) \in \mathcal F_{G[T_i]}$ and that $u_1, \ldots, u_{r-1} \in \hat S$.

Finally, we claim that for each $1 \leq i \leq n'$ the system
$$
\{A_{u,v} \cap A_i \ \ : \ \ 1 \leq u < v \leq |V(G)|, \ u \in T_i, \ v \in T_i \}
$$
is $(G[T_i],\lfloor n/2 \rfloor, \mathcal F_{G[T_i]})$-independent. To see this consider any $(v,S) \in \mathcal F_{G[T_i]}$. We wish to show that
\begin{equation}
\label{eq:smallindepsys}
\left( \bigcap_{u < v, u \in S, (u,v) \in E(G[T_i])} A_{u,v} \cap A_i\right) - \left( \bigcup_{w > v, w \in S, (v,w) \in E(G[T_i])} A_{v,w} \cap A_i\right) \neq \emptyset
\end{equation}
if $v > \min S$, and that
\begin{equation}
\label{eq:smallindepsys2}
A_i - \left( \bigcup_{w > v, w \in S, (v,w) \in E(G[T_i])} A_{v,w} \cap A_i \right) \neq \emptyset
\end{equation}
otherwise.
We must consider two cases:
\begin{enumerate}
\item There is $u < v$ such that $A_{u,v} = A_i$. By construction of $\mathcal F_{G[T_i]}$, there is some $S'$ with $S \cup \{ u\} \subset S'$ such that $(v,S') \in \mathcal F_G$. If $v > \min S$, then (\ref{eq:smallindepsys}) is satisfied by $(G,n,\mathcal F_G)$-independence of $\{A_{u,v}\}$, as $A_i$ will be one of the terms in the intersection in the definition of $(G,n,\mathcal F_G)$-independence and $v > u \geq \min S'$. If $v < \min S$, then (\ref{eq:smallindepsys2}) is satisfied since $u<v$ and thus $A_i$ is again one term in the intersection in the definition of $(G,n,\mathcal F_G)$-independence.
\item There is $w > v$ such that $A_{v,w} = [n] - A_i$. By construction of $\mathcal F_{G[T_i]}$, there is some $S'$ with $S \cup \{w\} \subset S'$ such that $(v,S') \in \mathcal F_G$. If $v > \min S$, then $v > \min S'$ as well, so (\ref{eq:smallindepsys}) is satisfied by $(G,n,\mathcal F_G)$-independence of $\{A_{u,v}\}$, as $[n]-A_i$ will be one of the terms in the union in the definition of $(G,n,\mathcal F_G)$-independence. If $v < \min S$, then regardless of whether $v < \min S'$, (\ref{eq:smallindepsys2}) is satisfied since again $[n] - A_i$ is one of the terms in the union in the definition of $(G,n,\mathcal F_G)$-independence.
\end{enumerate}
We have shown that each of (1), (2), (3) in the statement of the lemma hold.
 \end{proof}
Now we prove Theorem \ref{thm:reclowbnd}.
\begin{proof}
Let us associate $V(G)$ with the set $\{1, 2, \ldots, |V(G)|\}$. By Proposition \ref{prop:indepequiv} we have that there is a $(G,n,\mathcal F_{G,\phi})$-independent system.

We now use induction on $r$ and $n$ to prove the following claim: 
\begin{claim}
\label{clm:inductstep}
Suppose there exists a $(G,n,\mathcal F_G)$-independent system such that for any set of $r$ neighbors $\{u_1, \ldots, u_r\}$ of any vertex $v \in V(G)$, there is some $(v,S) \in \mathcal F_G$ with $\{u_1, \ldots, u_r\} \subseteq S$. Then $\chi(G) \leq 2^{{2n} + 2^{n/2^{r-2}}}$. 
\end{claim}
We first prove the base case $r = 2$. Let a $(G,n,\mathcal F_{G,\phi})$-independent system be denoted $\{A_{u,v}\}_{1 \leq u < v \leq |V(G)|}$. We now color $G$ by giving $v \in V(G)$ the set $B_v := \{A_{u,v}\}_{u \in N(v), u< v}$ as the color of $v$. In particular, the color of $v$ is a set of sets, each of which is a subset of $[n]$, meaning that the total number of colors is $2^{2^n} \leq 2^{2n + 2^{n/2^{2-2}}}$. We claim that this gives a proper coloring of $G$. To see this, consider any edge $(v,w) \in E(G)$, such that $v < w$. If there is no $u < v$ with $(u,v) \in E(G)$, then $B_v$ is the empty set, whereas $B_w$ at least contains $A_{v,w}$, so certainly $B_v \neq B_w$. Also, note that for any $u < v$ with $(u,v) \in E(G)$, we have that $(v,\{u,w\}) \in \mathcal F_G$, meaning that $A_{u,v} - A_{v,w} \neq \emptyset$. Therefore, $A_{v,w}$ does not contain any element in $B_v$ as a subset, so in particular $B_v$ cannot contain $A_{v,w}$. However, $B_w$ does contain $A_{v,w}$, meaning that $B_v$ and $B_w$ are distinct, as desired.

Now assume that for all $r < r_0$ and $n < n_0$, Claim \ref{clm:inductstep} holds. Suppose that there is a $(G,n_0,\mathcal F_G)$-independent system, where $\mathcal F_G$ is such that for any set of $r_0$ neighbors $\{u_1, \ldots, u_{r_0}\}$ of any vertex $v$, there is some $(v,S) \in \mathcal F_G$ with $\{u_1, \ldots, u_{r_0}\} \subseteq S$. By Lemma \ref{lem:inductr}, we may partition $V(G)$ into sets $T_1, \ldots, T_{n'}, T_{n'+1}$, with $n' \leq 2^{n_0-1}$, and such that properties (1) -- (3) in the lemma are satisfied. By the inductive hypothesis and conditions (2) and (3) of Lemma \ref{lem:inductr}, we have that for $1 \leq i \leq n'$,
$$
\chi(G[T_i]) \leq 2^{n_0 + 2^{\lfloor \frac{n_0}{2} \rfloor/2^{r_0-3}}} \leq 2^{n_0 + 2^{n_0/2^{r_0-2}}}.
$$
By condition (1) in Lemma \ref{lem:inductr}, we have that $\chi(G[T_{n'+1}]) = 1$. Therefore, we may color the vertices of $G$ by the product coloring of the unique $i$ such that any $v \in T_i$ and the coloring of $G[T_i]$, which gives:
$$
\chi(G) \leq 1 + n' \cdot 2^{n_0 + 2^{n_0/2^{r_0-2}}} \leq 2^{2n_0 + 2^{n_0/2^{r_0-2}}},
$$
where we have also used that only one color is needed for $T_{n'+1}$. This gives the desired result.
\end{proof}

For any fixed $r$, the function $n \mapsto 2^{2n + 2^{n/2^{r-2}}}$ is a strictly increasing continuous function of $n$ for $n \in \mathbb{R}^+$, so for any $m \in \mathbb{R}^+$ with $m \geq 2$, there is a unique $n$ with $m = 2^{2n + 2^{n/2^{r-2}}}$. Therefore, we may define $P_r(m)$ to be the inverse of the function $n \mapsto 2^{2n + 2^{n/2^{r-2}}}$ for $n \in \BR^+$. The next two lemmas establish an asymptotic form for $P_r(m)$ which will be useful in applying Theorem \ref{thm:reclowbnd}.

We let $\MW : \mathbb{R}^+ \ra \mathbb{R}^+$ be the Lambert $\MW$-function, defined as the inverse of the function $f(x) = xe^x$. (In particular, we are letting $\MW$ denote the principal branch of the Lambert $\MW$-function, restricted to the positive reals.) Then we have the following:
\begin{lemma}
  \label{lem:invmn}
If $m = 2^{2n + 2^{n/2^{r-2}}}$, then
$$
n = \frac{2 \ln m - 2^r \MW(2^{1-r}m^{2^{1-r}} \ln 2)}{4 \ln 2}.
$$
\end{lemma}
\begin{proof}
Define $\hat n = \frac{2 \ln m - 2^r \MW(2^{1-r} m^{2^{1-r}} \ln 2)}{4 \ln 2}$. Our goal is to show that $\hat n = n$.
Note that
$$
2^{\hat n} = \exp \left(\frac{\ln m}{2} - 2^{r-2} \MW(2^{1-r} m^{2^{1-r}} \ln 2) \right) = m^{1/2} \cdot \exp( - 2^{r-2} \MW(2^{1-r} m^{2^{1-r}} \ln 2)),
$$
which implies
\begin{equation}
  \label{eq:lw1}
2^{\hat n/2^{r-2}} = m^{1/2^{r-1}} \cdot \exp(-\MW(2^{1-r}m^{2^{1-r}}\ln 2)),
\end{equation}
and
\begin{equation}
  \label{eq:lw2}
2\hat n = \log_2 m - \frac{2^{r-1}}{\ln 2} \cdot \MW(2^{1-r} m^{2^{1-r}} \ln 2).
\end{equation}
Next we claim that
\begin{equation}
  \label{eq:lw3}
 m^{1/2^{r-1}} \cdot \exp(-\MW(2^{1-r}m^{2^{1-r}}\ln 2)) =  \frac{2^{r-1}}{\ln 2} \cdot \MW(2^{1-r} m^{2^{1-r}} \ln 2),
\end{equation}
which, by definition of $\MW$, is equivalent to
$$
\frac{m^{2^{1-r}} \cdot \ln 2}{2^{r-1}} = 2^{1-r} \cdot m^{2^{1-r}} \cdot \ln 2,
$$
which is trivially true. By (\ref{eq:lw1}), (\ref{eq:lw2}), and (\ref{eq:lw3}), $2^{\hat n /2^{r-2}} = \log_2m - 2 \hat n$, so $m = 2^{2 \hat n + 2^{\hat n /2^{r-2}}}$. Since the function $n \mapsto 2^{2n + 2^{n/2^{r-2}}}$ is a strictly increasing function of $n$ for any $r$, it must be the case that $\hat n = n$, as desired.
\end{proof}
 In the below lemma, recall that $\log$ denotes the base-2 logarithm while $\ln$ denotes the natural logarithm. 
\begin{lemma}
\label{lem:invw}
For a fixed $r$, we have that $P_r(m) = 2^{r-2} \cdot \log \log m + o(1)$ as $m \ra \infty$.
\end{lemma}
\begin{proof}
  By Lemma \ref{lem:invmn}, we have $P_r(m) = \frac{2 \ln m - 2^r \MW(2^{1-r} m^{2^{1-r}} \ln 2)}{4 \ln 2}$. It was shown in \cite{corless_lambert_1996} that $\MW(x) = \ln x - \ln \ln x + o(1)$ as $x \ra \infty$. 
For fixed $r$, note that $2^{1-r} m^{2^{1-r}} \ln 2 \ra \infty$ as $m \ra \infty$. Therefore,
\begin{eqnarray}
  && 4 \ln 2 \cdot P_r(m) \nonumber\\
  &=& 2 \ln m - 2^r \left( \ln \left( 2^{1-r} m^{2^{1-r}} \ln 2 \right) - \ln \ln \left(2^{1-r}m^{2^{1-r}} \ln 2 \right) + o(1) \right)\nonumber\\
                     &=& 2 \ln m - 2^r \left( (1-r) \ln 2 + 2^{1-r} \ln m + \ln \ln 2\right. \nonumber\\
  && \left.-\ln\left((1-r) \ln 2 + 2^{1-r} \ln m + \ln \ln 2\right) + o(1) \right)\nonumber\\
                     &=& 2^r(r-1) \ln 2 - 2^r \ln \ln 2 + o(1) + 2^r \ln ((1-r) \ln 2 + 2^{1-r} \ln m + \ln \ln 2) \nonumber\\
  \label{eq:breakln}
  &=& 2^r (r-1) \ln 2 - 2^r \ln \ln 2 + 2^r (1-r) \ln 2 + 2^r \ln \ln m + o(1) \\
                     &=& 2^r \ln \ln m - 2^r \ln \ln 2 + o(1) \nonumber\\
  &=& 2^r \ln \log m + o(1)\nonumber.
\end{eqnarray}
To arrive at equality (\ref{eq:breakln}) we have used that $\ln(2^{1-r} \ln m + (1-r) \ln 2 + \ln \ln 2) = \ln(2^{1-r} \ln m) + o(1)$ as $m \ra \infty$. From the above chain of equalities we then get that $P_r(m) = 2^{r-2} \frac{\ln \log m}{\ln 2} + o(1) = 2^{r-2} \log \log m + o(1)$, as desired.
\end{proof}
Using the previous results we next derive a lower bound on the chromatic number of $U(m,R,\delta)$ as well as of the subgraphs $W_r(m,\sigma,\delta)$ considered in the previous section. To do this we define two more families of graphs, denoted by $Y(m,\delta)$ and $Z(m,R,\delta)$ (see also Table \ref{tab:cgs}). Vertices of the graph $Y(m,\delta)$ are chains that grow arithmetically in size; in particular for $m,\delta \in \BN$, we define $Y(m,\delta)$ to be the subgraph of $U(m, 2\delta+1, \delta)$ induced by the set of vertices:
$$
V(Y(m,\delta)) := \{ \langle \alpha, A_1, \ldots, A_\delta \rangle : \alpha \in A_1 \subset \cdots \subset A_\delta \subset [m], \forall 1 \leq i \leq \delta, |A_i| = 2i + 1\}.
$$
Next, if moreover $R \geq 2\delta + 1$ we define $Z(m,R,\delta)$ to be the subgraph of $U(m,R,\delta)$ induced by the set of vertices:
\begin{equation}
  V(Z(m,R,\delta)) = \left\{ \langle \alpha, A_1, \ldots, A_\delta \rangle :
  \begin{array}{c}
    \alpha \in A_1 \subset \cdots \subset A_\delta \subset [m], |A_\delta| = R, \\
    \forall 1 \leq i \leq \delta -1, |A_i| = 2i + 1
  \end{array}
  \right\}.\nonumber
\end{equation}
By restriction to $Z(m,R,\delta)$, (\ref{eq:definephi}) defines a graph homomorphism $\phi : Z(m,R,\delta) \ra\\ Y(m,\delta-1)$. By Lemmas \ref{lem:tset} and \ref{lem:restindsub}, $\phi$ is complete.

We will derive a lower bound on $\chi(Z(m,R,\delta))$, which will then imply lower bounds on $\chi(U(m,R,\delta))$ and $\chi(W_r(m,\sigma,\delta))$, for appropriate choices of $r,\sigma$. We will need a few basic facts to do so. The following theorem gives a lower bound on the chromatic number of $Y(m,\delta)$:
\begin{theorem}[\cite{linial_locality_1992}, Theorem 2.1]
  \label{thm:shiftgraph}
For all $m, \delta$ with $\log^{(2\delta)}m \geq 1$, $\chi(Y(m,\delta)) \geq \log^{(2\delta)}m$.
\end{theorem}
Essentially equivalent statements of the above theorem can also be found in Theorem 2.7 of \cite{haramaty_deterministic_2016} and Theorem 7 of \cite{erdos_remarks_1964}.

The next lemma states a key property of the map $\phi : Z(m,R,\delta) \ra Y(m,\delta-1)$ that allows us to apply Theorem \ref{thm:reclowbnd}.
\begin{lemma}
  \label{lem:zinset}
  For any $m,R,\delta$ with $R \geq 2\delta+1$, the homomorphism $\phi : Z(m,R,\delta) \ra Y(m,\delta-1)$ satisfies the following property for $r \in \BN$ with $r \leq \frac{R-(2\delta-1)}{2}$. For any chains $\MA, \MB^{(1)}, \ldots, \MB^{(r)} \in V(Y(m,\delta-1))$, such that $(\MA, \MB^{(i)}) \in E(Y(m,\delta-1))$ for $1 \leq i \leq r$, there exists a pair $(\MA, S) \in \MF_{Y(m,\delta-1),\phi}$ such that for $1 \leq i \leq r$, $\MB^{(i)} \in S$.
\end{lemma}
\begin{proof}
  Let us write $\MA = \langle \alpha, A_1, \ldots, A_{\delta-1}\rangle$ and $\MB^{(i)} = \langle \beta^{(i)}, B_1^{(i)}, \ldots, B_{\delta-1}^{(i)} \rangle$ for $1 \leq i \leq r$. Since $B_{\delta-2}^{(i)} \subset A_{\delta-1}$ for each $i$, we have that $|B_{\delta-1}^{(i)} \backslash A_{\delta-1} | \leq 2$. Define
  $$
A_{\delta} = A_{\delta-1} \cup \bigcup_{1 \leq i \leq r} B_{\delta-1}^{(i)} \backslash A_{\delta-1},
$$
so that $|A_\delta| \leq 2\delta-1 + 2r \leq R$. If necessary, add a few arbitrary elements of $[m]$ to $A_\delta$ so that its size is exactly $R$. Now define $\tilde{\MA} := \langle \alpha, A_1, \ldots, A_{\delta-1}, A_\delta \rangle \in V(Z(m,R,\delta))$. Now indeed each $\MB^{(i)} \in \phi(N(\tilde{\MA}))$, so we may take $S = \phi(N(\tilde{\MA}))$, completing the proof.
\end{proof}

We may now derive a lower bound on $\chi(Z(m,R,\delta))$:
\begin{theorem}
\label{thm:stronglowbnd}
We have $\chi(Z(m,R,\delta)) > 2^{\frac{R-2\delta-4}{2}} \log^{(2\delta)}m + o(1)$ as $m \ra \infty$.
\end{theorem}
\begin{proof}
  By Theorem \ref{thm:shiftgraph}, we have that $\chi(Y(m,\delta)) \geq \log^{(2\delta-2)}m$. Let us write $n = \chi(Z(m,R,\delta))$. By Theorem \ref{thm:reclowbnd} with $r = \left\lfloor \frac{R - (2\delta -1)}{2} \right\rfloor$ and Lemma \ref{lem:zinset}, we have that $\log^{(2\delta-2)}m \leq \chi(Y(m,\delta)) \leq 2^{2n + 2^{n/2^{r-2}}}$. By Lemma \ref{lem:invw}, we get that
 
$$
\chi(Z(m,R,\delta)) \geq 2^{\left \lfloor \frac{R-2\delta+1}{2} \right\rfloor- 2} \cdot \log^{(2\delta)}m + o(1),
$$
as $m \ra \infty$, as desired.
\end{proof}
Theorem \ref{thm:stronglowbnd} implies the following: given $m,R,\delta$, there exists a graph $G$ (namely, $Z(m,R,\delta)$) with a proper coloring with $m$ colors such that there are at most $R$ colors in the distance-$\delta$ neighborhood of each vertex of $G$, such that $\chi(G) >  2^{\frac{R - 2\delta - 4}{2}} \cdot \log^{(2\delta)}m + o(1)$. This generalizes Theorem 2.3 in \cite{erdos_coloring_1986} and Theorem 6 in \cite{szegedy_locality_1993}.

Since $Z(m,R,\delta)$ is a subgraph of $U(m,R,\delta)$, the following corollary of Theorem \ref{thm:stronglowbnd} is immediate:
\begin{corollary}
  \label{cor:umrdlowbnd}
We have $\chi(U(m,R,\delta)) > 2^{\frac{R-2\delta-4}{2}} \log^{(2\delta)}m + o(1)$ as $m \ra \infty$.
\end{corollary}
We next embed $Z(m/2,R,\delta)$ in $W_r(m,\sigma,\delta)$ for appropriate choices of $R,\sigma,r$, allowing us to derive a lower bound on $\chi(W_r(m,\sigma,\delta))$ in the below corollary.
\begin{corollary}
  \label{cor:wrlowbnd}
  For $m, r, \sigma, \delta \in \BN$ with $r,\sigma \geq 2$, we have $$\chi(W_r(m,\sigma,\delta)) > 2^{\frac{r^{\sigma+\delta-2}(r-1)-5}{2}} \log^{(2\delta)} m  + o(1)$$ as $m \ra \infty$.
\end{corollary}
\begin{proof}
  Fix $r, \sigma, \delta$. Let $\tilde r = r^{\sigma + \delta-1} - r^{\sigma + \delta-2}$.
  For sufficiently large $m$, we define a graph homomorphism $\psi : Z(m/2, 2\delta-1 + \tilde r, \delta) \ra W_r(m, \sigma, \delta)$, as follows. First choose sets $C_1 \subseteq C_2 \subseteq \cdots \subseteq C_{\delta-1} \subseteq \{ m/2+1, \ldots, m\}$ such that $|C_i| = r^{\sigma + i-1} - (2i+1)$ for $1 \leq i \leq \delta-1$. This is certainly possible for large enough $m$.

  Now consider a chain $\MA = \langle \alpha, A_1, \ldots, A_\delta \rangle \in V(Z(m/2, 2\delta-1 + \tilde r, \delta))$. We map it to the chain $\psi(\MA) := \langle \alpha ,A_1 \cup C_1, A_2 \cup C_2, \ldots, A_{\delta-1} \cup C_{\delta-1}, A_{\delta} \cup C_{\delta-1} \rangle$. From this construction, as well as the fact that $|A_\delta \cup C_{\delta-1} | = |A_{\delta-1}| + | C_{\delta-1} | + |A_\delta \backslash A_{\delta-1}| = r^{\sigma + \delta-2} + \tilde r = r^{\sigma + \delta-1}$, we have that $\psi(\MA) \in V(W_r(m,\sigma,\delta))$. It is also evident from the construction that $\psi$ is indeed a graph homomorphism, meaning that $\chi(Z(m/2, 2\delta-1 +\tilde r, \delta)) \leq \chi(W_r(m,\sigma,\delta))$. By Theorem \ref{thm:stronglowbnd}, it follows that
  $$
\chi(W_r(m,\sigma,\delta)) \geq 2^{\frac{\tilde r - 5}{2}} \log^{(2\delta)}(m/2) + o(1) \geq 2^{\frac{r^{\sigma+\delta-2}(r-1)-5}{2}} \log^{(2\delta)} m + o(1),
$$
as desired.
\end{proof}

\section{Conclusion}
Now we compare the necessary conditions needed on the homomorphism $\phi : H \ra G$ to obtain lower (Theorem \ref{thm:reclowbnd}) and upper (Lemma \ref{lem:makegraphindep}) bounds on $\chi(H)$ if we already know $\chi(G)$. We assume that $\phi$ is complete. The requirement in Theorem \ref{thm:reclowbnd} is that for any $v \in G$ and neighbors $u_1, \ldots, u_r \in V(G)$ of $v$, there is some $(v,S) \in \mathcal F_{G,\phi}$ such that $u_1, \ldots, u_r \in S$. The requirement in Lemma \ref{lem:makegraphindep} is that for any $(v,S) \in \mathcal F_G$, we have that $|S| \leq r$. Note that in the specific case where $G = K_m$ and $H = U(m,r+1,1)$, these conditions are equivalent. However, for other values of $G$ (in particular, say $G = Z(m, R,\delta)$), we will have that for any $(v,S) \in \mathcal F_{G,\phi}$, the size of $S$ is much larger than the maximum number of neighbors $u_1, \ldots, u_r$ of $v$ for which there is some $(v,S) \in \mathcal F_{G,\phi}$ such that $u_1, \ldots, u_r \in S$. The difference between these two conditions creates the gap in the bound
\begin{equation}
  \label{eq:gapbnds}
2^{\frac{r^{\sigma+\delta-2}(r-1)-5}{2}} \log^{(2\delta)}m + o_m(1) \leq \chi(W_r(m,\sigma,\delta)) \leq 2^{2^{2+r^{\sigma+\delta-2}(r+1)}}  \cdot \log^{(2\delta)}m
\end{equation}
on $\chi(W_r(m,\sigma,\delta))$, from  Theorem \ref{thm:uppbndexp} and Corollary \ref{cor:wrlowbnd}, as well as the gap in the bound
\begin{equation}
  \label{eq:gapumr}
  2^{\frac{R - 2\delta - 4}{2}} \log^{(2\delta)}m + o_m(1) \leq \chi(U(m,R,\delta)) \leq 2^{2^{2 + 2R}} \log^{(2\delta)}m + o_m(1),
\end{equation}
from Theorem \ref{thm:uppbndgen} and Corollary \ref{cor:umrdlowbnd}. We remark that prior to our work the best known lower bound on $\chi(W_r(m,\sigma,\delta))$ and $\chi(U(m,R,\delta))$ was $\log^{(2\delta)}m$, though it is still of great interest to close the gap in (\ref{eq:gapbnds}) and (\ref{eq:gapumr}). 

The lower bound of $2^{\frac{R-2\delta-4}{2}} \log^{(2\delta)}m + o(1)$ on $\chi(U(m,R,\delta))$ from Theorem \ref{thm:stronglowbnd} casts doubt on the possibility that there is a constant-length encryption scheme with no error as described in Section 1. Recall that improving the upper bound on $\chi(U(m,R,\delta))$ from \cite{haramaty_deterministic_2016} by replacing the term $2^{O(R)}$ by a constant (independent of $R$) is sufficient for finding such a compression scheme. We have shown that it is not possible to do this for small values of $\delta$, whereas previously such a result was only known for the case $\delta = 1$ \cite{erdos_coloring_1986,szegedy_locality_1993}. 

\section{Acknowledgements}
I am grateful to Madhu Sudan for calling my attention to this problem and for helpful discussions, to Aaron Potechin for helpful suggestions and for pointing out a few minor errors in the proofs, and to an anonymous referee for helpful suggestions.

\appendix
\section{Proof of Theorem \ref{thm:uppbndexp}}
\label{apx:pfs}
Lemma \ref{lem:boundtset} uses the exponential growth of the chains $\MA \in V(W_r(m,\sigma,\delta))$ to obtain a better upper bound on $|\phi(N(\MA)) |$ than the one in Lemma \ref{lem:boundtsetgen}.
\begin{lemma}
\label{lem:boundtset}
For $\mathcal A \in V(W_r(m,\sigma,\delta))$, we have that $|\phi(N(\mathcal A))| \leq 2^{r^{\sigma+\delta-2}(r-1)}$.
\end{lemma}
\begin{proof}
  Exactly as in Lemma \ref{lem:boundtsetgen},
using the fact that for all $\langle \beta, B_1, \ldots, B_{\delta-1} \rangle \in \phi(N(\mathcal A))$, we must have $\beta \in A_1$, $\alpha \in B_1$, $B_1 \subseteq A_2$, and $A_{i-1} \subseteq B_i \subseteq A_{i+1}$ for $2 \leq i \leq \delta-1$, the number of elements in $\phi(N(\mathcal A))$ is at most
\begin{equation}
  2^{|A_1| + \cdots + |A_\delta| - (|A_1| + \cdots + |A_{\delta-2}|)} = 2^{r^{\sigma + \delta-2}(1+r)}.\nonumber
\end{equation}
\end{proof}
Note that the bound in Lemma \ref{lem:boundtset} can be written as $| \phi(N(\MA)) | \leq 2^{sz(\MA) \cdot \frac{r+1}{r}}$ which is better than the bound $|\phi(N(\MA)) | \leq 2^{2\cdot sz(\MA) }$ from Lemma \ref{lem:boundtsetgen}. 

\begin{proof}[Proof of Theorem \ref{thm:uppbndexp}]
  We use induction on $\delta$. For the base case $\delta = 0$, we have that $W_r(m,\sigma,\delta) = K_m$, the complete graph on $m$ vertices. Then clearly $\chi(K_m) = m = \log^{(0)}m$.

  Now suppose the result is true for $\delta-1$. 
  Let $G = W_r(m,\sigma,\delta-1)$, so that $\chi(G) \leq 2^{2^{2+r^{\sigma+\delta-2}(r+1)}} \cdot \log^{(2\delta-2)}m$. Let $H = W_r(m,\sigma,\delta)$. Next, (\ref{eq:definephi}) defines a graph homomorphism $\phi : H \ra G$, by restricting to the induced subgraph $H = W_r(m,\sigma,\delta)$ of $U(m,r^{\sigma + \delta-1}, \delta)$. 
By Lemma \ref{lem:boundtset}, for each pair $(v,S) \in \MF_{G,\phi}$, we have that $|S| \leq 2^{r^{\sigma+\delta-2}(r+1)}$.

Let $\rho = 2^{r^{\sigma + \delta-2}(r+1)}$. By Theorem \ref{thm:homoupp}, we have that $$\chi(W_r(m,\sigma,\delta)) \leq \lceil \rho 2^\rho \cdot  \log \log \chi(W_r(m,\sigma,\delta-1))\rceil.$$
If $2^{2^{2+r^{\sigma+\delta-3}(r+1)}} \leq \log^{(2\delta-2)}m$, then the above equation implies that
\begin{eqnarray*}
  \chi(W_r(m,\sigma,\delta)) &\leq& \lceil \rho 2^\rho (1 + \log^{(2\delta)}m)\rceil \\
  &\leq&
         2^{2 + r^{\sigma+\delta-2}(r+1)} \cdot 2^{2^{r^{\sigma+\delta-2}(r+1)}} \cdot \log^{(2\delta)}m  \\
  &\leq& 2^{2^{2+r^{\sigma+\delta-2}(r+1)}} \cdot \log^{(2\delta)}m.
\end{eqnarray*}

If $2^{2^{2+r^{\sigma+\delta-3}(r+1)}} > \log^{(2\delta-2)}m$, then since the function $\frac{\log \log n}{n}$ is a decreasing function of integers $n$ for $n \geq 4$, we have that
$$
\frac{\log^{(2\delta)}m}{\log^{(2\delta-2)}m} > \frac{2 + r^{\sigma+\delta-3}(r+1)}{2^{2^{2+r^{\sigma+\delta-3}(r+1)}}},
$$
which implies that
$$
\frac{\log^{(2\delta)}m \cdot 2^{2^{2+r^{\sigma+\delta-2}(r+1)}}}{\log^{(2\delta-2)}m \cdot 2^{2^{2+r^{\sigma+\delta-3}(r+1)}}} \geq \frac{\left(2+r^{\sigma+\delta-3}(r+1)\right) \cdot 2^{2^{2+r^{\sigma+\delta-2}(r+1)}}}{\left(2^{2^{2+r^{\sigma+\delta-3}(r+1)}}\right)^2} > 1,
$$
where the last inequality follows from $2 + r^{\sigma + \delta-2} (r+1) \geq 3 + r^{\sigma + \delta - 3}(r+1)$,
so we may use the trivial bound
$$
\chi(W_r(m,\sigma,\delta)) \leq \chi(W_r(m,\sigma,\delta-1)) \leq \log^{(2\delta - 2)}m \cdot 2^{2^{2 + r^{\sigma + \delta -3}(r+1)}}\leq 2^{2^{2+r^{\sigma+\delta-2}(r+1)}} \cdot \log^{(2\delta)}m.
$$
\end{proof}



\end{document}